\documentclass[12pt]{l4dc2021_MODDED} 

\makeatletter
\def\set@curr@file#1{\def\@curr@file{#1}} 
\makeatother

\usepackage[load-configurations=version-1]{siunitx} 

\title[Approximate midpoint policy iteration for linear quadratic control]{Approximate Midpoint Policy Iteration for Linear Quadratic Control}
\usepackage{times}
\usepackage{algorithm}
\usepackage[noend]{algorithmic}

\usepackage{mathtools}
\usepackage{enumitem}
\usepackage{wrapfig}

\newcommand{\EE}{\mathbb{E}}
\newcommand{\RR}{\mathbb{R}}
\newcommand{\bbS}{\mathbb{S}}

\newcommand{\cD}{\mathcal{D}}

\newcommand{\cH}{\mathcal{H}}

\newcommand{\cK}{\mathcal{K}}

\newcommand{\cN}{\mathcal{N}}

\newcommand{\cQ}{\mathcal{Q}}
\newcommand{\cR}{\mathcal{R}}

\newcommand{\cU}{\mathcal{U}}

\newcommand{\cX}{\mathcal{X}}

\newcommand{\tp}{\intercal}

\newcommand{\Let}{\coloneqq}

\newcommand{\eps}{\varepsilon}

\DeclareMathOperator{\vect}{vec}
\DeclareMathOperator{\mat}{mat}
\DeclareMathOperator{\svec}{svec}
\DeclareMathOperator{\smat}{smat}
\DeclareMathOperator{\Tr}{Tr}

\newcommand{\dlyap}{\texttt{DLYAP}}

\newcommand{\lstdq}{\texttt{LSTDQ}}

\newcommand{\rollout}{\texttt{ROLLOUT}}

\author{%
 \Name{Benjamin Gravell} \Email{benjamin.gravell@utdallas.edu}\\
 \addr University of Texas at Dallas
 \AND
 \Name{Iman Shames} \Email{iman.shames@unimelb.edu.au}\\
 \addr University of Melbourne%
 \AND
 \Name{Tyler Summers} \Email{tyler.summers@utdallas.edu}\\
 \addr University of Texas at Dallas%
}

\begin{document}

\maketitle

\begin{abstract}%
We present a midpoint policy iteration algorithm to solve linear quadratic optimal control problems in both model-based and model-free settings. The algorithm is a variation of Newton's method, and we show that in the model-based setting it achieves \emph{cubic} convergence, which is superior to standard policy iteration and policy gradient algorithms that achieve quadratic and linear convergence, respectively. We also demonstrate that the algorithm can be approximately implemented without knowledge of the dynamics model by using least-squares estimates of the state-action value function from trajectory data, from which policy improvements can be obtained. With sufficient trajectory data, the policy iterates converge cubically to approximately optimal policies, and this occurs with  the  same  available  sample  budget  as  the  approximate standard policy iteration. Numerical experiments demonstrate effectiveness of the proposed algorithms.
\end{abstract}


\section{Introduction}
With the recent confluence of reinforcement learning and data-driven optimal control, there is renewed interest in fully understanding convergence, sample complexity, and robustness in both ``model-based'' and ``model-free'' algorithms. Linear quadratic problems in continuous spaces provide benchmarks where strong theoretical statements can be made.
In practice, it is often difficult or impossible to develop a model of a system from first-principles. In this case, one may use so-called ``model-based'' system identification methods which attempt to estimate a model of the dynamics from observed sample data, then solve the Riccati equation using the identified system matrices. An approximately optimal control policy is then computed assuming certainty-equivalence \cite{Mania2019CertaintyEC, oymak2019non, coppens2020sample} or using robust control approaches to explicitly account for model uncertainty \cite{dean2018regret, Dean2019, gravell2020pmlr, coppens2020data}. The analyses in these recent works has focused on providing finite-sample performance/suboptimality guarantees.

As an alternative, so-called ``model-free'' methods may also be used, which do not attempt to learn a model of the dynamics. The category of policy optimization methods which directly attempt to optimize the control policy, including policy gradient, has received significant attention recently for standard LQR \cite{Fazel2018, bu2020lqr}, multiplicative-noise LQR \cite{gravell2021tac}, Markov jump LQR \cite{jansch2020jump}, and LQ games related to $\mathcal{H}_\infty$ robust control \cite{zhang2019policy, bu2019games}.

Between the fully model-based system identification approaches and the fully model-free policy optimization approaches lies another category of methods, which we denote as value function approximation methods. These methods attempt to estimate value functions then compute policies which are optimal with respect to these value functions. This class of methods includes approximate dynamic programming, exemplified by approximate value iteration, which estimates state-value functions, and approximate policy iteration, which estimates state-action value functions. In particular, for LQR problems, approximate policy iteration was studied by \cite{bradtke1994adaptive,  krauth2019finite} and by \cite{Fazel2018, bu2020lqr} under the guise of a quasi-Newton method.
For LQ games, approximate policy iteration was studied by \cite{altamimi2007model} under the guise of Q-learning, and by \cite{luo2020, gravell2020css}.
Note that approximate policy iteration is sometimes called quasi-Newton or Q-learning.

In stochastic optimal control, the functional Bellman equation gives a sufficient and necessary condition for optimality of a control policy (\cite{bellman1959dynamic}).
It has been long-known, but perhaps underappreciated, that application of Newton's method to find the root of the functional Bellman equation in stochastic optimal control is equivalent to the dynamic programming algorithm of policy iteration (\cite{puterman1979, madani2002policy}). 
In this most general setting, conditions for and rates of convergence are available (\cite{puterman1979, madani2002policy}), but may be difficult or impossible to verify in practice. Furthermore, even representing the value functions and policies and executing the policy iteration updates may be intractable. This motivates the basic approximation of such problems by linear dynamics and quadratic costs over finite-dimensional, infinite-cardinality state and action spaces. 
In linear-quadratic problems, the Bellman equation becomes a matrix algebraic Riccati equation, and application of the Newton method to the Riccati equation yields the well-known Kleinman-Hewer algorithm.\footnote{\cite{kleinman1968iterative} introduced this for continuous-time systems, and \cite{hewer1971iterative} studied it for discrete-time systems.}
The Newton method has many variations devised to improve the convergence rate and information efficiency, including higher-order methods 
(such as Halley (\cite{cuyt1985computational}), super-Halley (\cite{gutierrez2001super}), and Chebyshev (\cite{argyros1993results})),
and multi-point methods (\cite{traub1964iterative}), which compute derivatives at multiple points and of which the midpoint method is the simplest member.
Some of these have been applied to solving Riccati equations by \cite{anderson1978second, guo2000newton, damm2001newton, freiling2004, hernandez2018solving}, but without consideration of the situation when the dynamics are not perfectly known.
Our main contributions are:
\vspace{-0.5\baselineskip}
\begin{enumerate}[noitemsep]
    \item We present a midpoint policy iteration algorithm to solve linear quadratic optimal control problems when the dynamics are both known (Algorithm \ref{algo:exact_mpi}) and unknown (Algorithm \ref{algo:approximate_mpi}).
    \item We demonstrate that the method converges, and does so at a faster \emph{cubic} rate than standard policy iteration or policy gradient, which converge at quadratic and linear rates, respectively.
    \item We show that approximate midpoint policy iteration converges faster in the model-free setting even with the same available sample budget as the approximate standard policy iteration.
    \item We present numerical experiments that illustrate and demonstrate the effectiveness of the algorithms and provide an open-source implementation to facilitate their wider use.
\end{enumerate}

\section{Preliminaries}
\begin{center}
  \begin{tabular}{p{2cm} p{12.5cm}} 
  Symbol & Meaning  \\
  \hline
    {$\RR^{n \times m}$} & {Space of real-valued $n \times m$ matrices} \\
    {$\bbS^{n}$} & {Space of symmetric real-valued $n \times n$ matrices} \\
    {$\bbS^{n}_{+}$} & {Space of symmetric real-valued positive semidefinite $n \times n$ matrices} \\
    {$\bbS^{n}_{++}$} & {Space of symmetric real-valued strictly positive definite $n \times n$ matrices} \\
    {$\rho(M)$} & {Spectral radius (greatest magnitude of an eigenvalue) of a square matrix $M$} \\
    {$\lVert M \rVert$} & {Spectral norm (greatest singular value) of a matrix $M$} \\
    {$\lVert M \rVert_F$} & {Frobenius norm (Euclidean norm of the vector of singular values) of a matrix $M$} \\
    {$M \otimes N$} & {Kronecker product of matrices $M$ and $N$} \\
    {$\vect(M)$} & {Vectorization of matrix $M$ by stacking its columns} \\
    {$\mat(v)$} & {Matricization of vector $v$ such that $\mat(\vect(M))=M$} \\
    {$\svec(M)$} & {Symmetric vectorization of matrix $M$ by stacking columns of the upper triangular part, including the main diagonal, with off-diagonal entries multiplied by $\sqrt{2}$ such that $\lVert M \rVert_F^2 = \svec(M)^\tp \svec(M)$} \\
    {$\smat(v)$} & {Symmetric matricization of vector $v$ i.e. inverse operation of $\svec(\cdot)$ such that $\smat(\svec(M))=M$} \\
    {$M \succ (\succeq) \ 0$} & {Matrix $M$ is positive (semi)definite} \\
    {$M \succ (\succeq) \ N$} & {Matrix $M$ succeeds matrix $N$ as $M-N \succ (\succeq) \ 0$ } \\
  \end{tabular}
\end{center}

The infinite-horizon average-cost time-invariant linear quadratic regulator (LQR) problem is
\begin{alignat}{2}  \label{eq:LQR}
    &\underset{{\pi \in \Pi}}{\text{minimize}} \quad && \lim_{T \to \infty} \frac{1}{T} \mathbb{E}_{x_0, w_t} \sum_{t=0}^T \begin{bmatrix} x_t \\ u_t \end{bmatrix}^\tp \begin{bmatrix} Q_{xx} & Q_{xu} \\ Q_{ux}  & Q_{uu} \end{bmatrix} \begin{bmatrix} x_t \\ u_t \end{bmatrix}, \\
    &\text{subject to}                         \quad && x_{t+1} = A x_t +  B u_t + w_t, \nonumber
\end{alignat}
where $x_t \in \RR^n$ is the system state, $u_t \in \RR^m$ is the control input, and $w_t$ is i.i.d. process noise with zero mean and covariance matrix $W$. The state-to-state system matrix $A \in \RR^{n \times n}$ and input-to-state system matrix $B \in \RR^{n \times m}$ may or may not be known; we present algorithms for both settings.
The optimization is over the space $\Pi$ of (measurable) history dependent feedback policies $\pi = \{ \pi_t \}_{t=0}^{\infty}$ with $u_t = \pi_t(x_{0:t}, u_{0:t-1})$. 
The penalty weight matrix
\begin{align*}
    Q = \begin{bmatrix} Q_{xx} & Q_{xu} \\ Q_{ux}  & Q_{uu} \end{bmatrix} \in \bbS^{n+m}
\end{align*}
has blocks $Q_{xx}$, $Q_{uu}$, $Q_{xu}=Q_{ux}^\tp$ which quadratically penalize deviations of the state, input, and product of state and input from the origin, respectively.
We assume
the pair $(A, B)$ is stabilizable,
the pair $(A, Q_{xx}^{1/2})$ is detectable, and
the penalty matrices satisfy the definiteness condition $Q \succ 0$,
in order to ensure feasibility of the problem (see \cite{anderson2007optimal}).
An LQR problem is fully described by the tuple of problem data $(A, B, Q)$, which are fixed after problem definition.
In general, operators denoted by uppercase calligraphic letters depend on the problem data $(A, B, Q)$, but we will not explicitly notate this for brevity; dependence on other parameters will be denoted explicitly by functional arguments.
We index over time in the evolution of a dynamical system with the letter $t$, and index over iterations of an algorithm with the letter $k$.

Dynamic programming can be used to show that the optimal policy that solves \eqref{eq:LQR} is linear state-feedback
\begin{align*}
    u_t = K x_t ,
\end{align*}
where the gain matrix $K=\cK(P)$ is expressed through the linear-fractional operator $\cK$
\begin{align*}
    \cK(P) 
    &\Let
    -\left( Q_{uu} + B^\tp P B \right)^{-1}(Q_{ux} + B^\tp P A) ,
\end{align*}
and $P$ is the optimal value matrix found by solving the algebraic Riccati equation (ARE)
\begin{align}
    \cR(P) = 0, \label{eq:ricc}
\end{align}
where $\cR$ is the quadratic-fractional Riccati operator
\begin{align*}
    \cR(P) 
    &\Let
    -P + Q_{xx} + A^\tp P A - (Q_{xu} + A^\tp P B) \left( Q_{uu} + B^\tp P B \right)^{-1}(Q_{ux} + B^\tp P A) .
\end{align*}
The optimal gain and value matrix operators can be expressed more compactly as
\begin{align*}
    \cK(P) 
    &=
    -\cH_{uu}^{-1}(P) \cH_{ux}(P) , \\
    \cR(P) 
    &=
    -P + \cH_{xx}(P) - \cH_{xu}(P) \cH_{uu}^{-1}(P) \cH_{ux}(P)
\end{align*}
where $\cH$ is the state-action value matrix operator
\begin{align*}
    \cH(P)
    =
    \begin{bmatrix}
    \cH_{xx}(P) & \cH_{xu}(P) \\
    \cH_{ux}(P) & \cH_{uu}(P)
    \end{bmatrix}
    \Let
    Q
    +
    \begin{bmatrix}
    A & B
    \end{bmatrix}^\tp
    P
    \begin{bmatrix}
    A & B
    \end{bmatrix}.
\end{align*}
The discrete-time Lyapunov equation with matrix $F$ and symmetric matrix $S$ is
\begin{align*}
    X = F^\tp X F + S,
\end{align*}
whose solution we denote by $X = \dlyap(F, S)$, which is unique if $F$ is Schur stable.

\subsection{Derivatives of the Riccati operator}
The first total derivative
\footnote{In infinite dimensions, the first total derivative is called the \emph{Fr\'echet} derivative, and the first directional derivative is called the \emph{Gateaux} derivative. As we are only considering finite-dimensional systems, we do not need the full generality of these objects.} 
of the Riccati operator evaluated at point $P \in \bbS^{n}$ is denoted as $\cR^\prime(P) \in \bbS^{n} \times \bbS^{n}$. With a slight abuse of notation, the first directional derivative of the Riccati operator evaluated at point $P$ in direction $X$ is denoted as $\cR^\prime(P, X) \in \bbS^{n}$.
Computation of the first directional derivative is straightforward and follows e.g. the derivation given by \cite{luo2020}. The general limit definition of this derivative is \begin{align*}
    \cR^\prime(P, X) 
    \Let  
    \lim_{\eps \to 0} \frac{\cR(P + \eps X) - \cR(P)}{\eps}
    =
    \left. \frac{d\cR(P + \eps X)}{d\eps} \right|_{\eps=0}
\end{align*}
Notice that since $\cR: \bbS^{n} \to \bbS^{n}$ it follows that $\cR^\prime(\cdot, \cdot): \bbS^{n} \times \bbS^{n} \to \bbS^{n}$.
In evaluating the first directional derivative, it will be useful note that
\begin{align*}
    \cH(P + \eps X)
    =
    Q
    +
    \begin{bmatrix}
    A & B
    \end{bmatrix}^\tp
    (P + \eps X)
    \begin{bmatrix}
    A & B
    \end{bmatrix}.
\end{align*}
The first derivative is then
\begin{align*}
    \cR^\prime(P, X) 
    &=
    \left. \frac{d\cR(P + \eps X)}{d\eps} \right|_{\eps=0} \\
    &=
    \left. \frac{d}{d\eps} 
    \left[
    -(P + \eps X) + \cH_{xx}(P + \eps X) - \cH_{xu}(P + \eps X) \cH_{uu}^{-1}(P + \eps X) \cH_{ux}(P + \eps X)
    \right] \right|_{\eps=0}
\end{align*}
At this point it will be useful to evaluate the following expressions:
\begin{align*}
    \left. \frac{d \cH(P + \eps X)}{d\eps} \right|_{\eps=0}
    &=
    \left. \frac{d}{d\eps} \left[ 
    Q
    +
    \begin{bmatrix}
    A & B
    \end{bmatrix}^\tp
    (P + \eps X)
    \begin{bmatrix}
    A & B
    \end{bmatrix}
    \right] \right|_{\eps=0} \\
    &=
    \begin{bmatrix}
    A & B
    \end{bmatrix}^\tp
    X
    \begin{bmatrix}
    A & B
    \end{bmatrix},
\end{align*}
and
\begin{align*}
    \left. \frac{d \cH_{uu}^{-1}(P + \eps X)}{d\eps} \right|_{\eps=0}
    &=
    -\cH_{uu}^{-1}(P) \left. \frac{d \cH_{uu}(P + \eps X)}{d\eps} \right|_{\eps=0} \cH_{uu}^{-1}(P) \\
    &=
    -\left( Q_{uu} + B^\tp P B \right)^{-1} \left( B^\tp X B \right)  \left( Q_{uu} + B^\tp P B \right)^{-1} ,
\end{align*}
where we used the rule for a derivative of a matrix inverse e.g. as in \cite{selby1974crc}.
Continuing with the first derivative,
\begin{align}
    \cR^\prime(P, X) 
    &=
    \left. \frac{d}{d\eps} 
    \left[
    -(P + \eps X) + \cH_{xx}(P + \eps X) - \cH_{xu}(P + \eps X) \cH_{uu}^{-1}(P + \eps X) \cH_{ux}(P + \eps X)
    \right] \right|_{\eps=0} \nonumber \\
    &=
    - \left. \frac{d}{d\eps} \left[ P + \eps X \right] \right|_{\eps=0}
    +
    \left. \frac{d}{d\eps} \left[\cH_{xx}(P + \eps X) \right] \right|_{\eps=0} \nonumber \\
    &\quad -
    \left. \frac{d}{d\eps} \left[ \cH_{xu}(P + \eps X) \right] \right|_{\eps=0} \cH_{uu}^{-1}(P) \cH_{ux}(P) \nonumber \\
    &\quad -
    \cH_{xu}(P) \left. \frac{d}{d\eps} \left[ \cH_{uu}^{-1}(P + \eps X) \right] \right|_{\eps=0} \cH_{ux}(P) \nonumber \\
    &\quad -
    \cH_{xu}(P) \cH_{uu}^{-1}(P) \left. \frac{d}{d\eps} \left[ \cH_{ux}(P + \eps X) \right] \right|_{\eps=0} \nonumber \\
    &=
    -X + A^{\tp} X A \nonumber \\
    & \quad - A^{\tp} X B\left(Q_{uu} + B^{\tp} P B\right)^{-1} (Q_{ux} + B^{\tp} P A) \nonumber \\
    & \quad - (Q_{xu} + A^{\tp} P B) \left( Q_{uu} + B^{\tp} P B\right)^{-1} B^{\tp} X A \nonumber \\
    & \quad +A^{\tp} P B\left(Q_{uu}+B^{\tp} P B\right)^{-1} B^{\tp} X B\left(Q_{uu} + B^{\tp} P B\right)^{-1} B^{\tp} P A \label{eq:dR}
\end{align}
where we used the product rule for matrix derivatives.

\subsection{Identities}
Considering two symmetric matrices $P, X$ and the related gains
\begin{align*}
K &= \cK(P) = -\left( Q_{uu} + B^\tp P B \right)^{-1} (Q_{ux} + B^\tp P A), \\
L &= \cK(X) = -\left( Q_{uu} + B^\tp X B \right)^{-1} (Q_{ux} + B^\tp X A),    
\end{align*}
we have
\begin{align*}
    Q_{ux} + Q_{uu} K &= - B^\tp P (A+BK), \\
    K^\tp Q_{ux} + K^\tp Q_{uu} K &= - (BK)^\tp P (A+BK) .
\end{align*}
Thus the Riccati operator $\cR(P)$ can be rewritten as
\begin{align*}
    \cR(P) 
    &= - P + Q_{xx} + Q_{xu} K + A^\tp P (A + BK) \\
    &= - P + Q_{xx} + Q_{xu} K - (BK)^\tp P (A + BK) + (A + BK)^\tp P (A + BK) \\
    &= - P +
    \begin{bmatrix}
    I \\ K
    \end{bmatrix}^\tp
    \begin{bmatrix}
    Q_{xx} & Q_{xu} \\
    Q_{ux} & Q_{uu}
    \end{bmatrix}
    \begin{bmatrix}
    I \\ K
    \end{bmatrix}
    + (A + BK)^\tp P (A + BK) \\
    &= -P +
    \begin{bmatrix}
    I \\ K
    \end{bmatrix}^\tp
    \left(
    Q + \begin{bmatrix} A & B \end{bmatrix}^\tp P \begin{bmatrix} A & B \end{bmatrix}
    \right)
    \begin{bmatrix}
    I \\ K
    \end{bmatrix}
    ,
\end{align*}
and the derivative $\cR^\prime(X, P)$ can be written using \eqref{eq:dR} as
\begin{align*}
    \cR^\prime(X, P) &= -P + (A+BL)^\tp P (A+BL).
\end{align*}
and we have the identity 
\begin{align}
    \cR(P) - \cR^\prime(X, P)
    &=
    \begin{bmatrix}
    I \\ K
    \end{bmatrix}^\tp
    Q
    \begin{bmatrix}
    I \\ K
    \end{bmatrix} + (A + BK)^\tp P (A + BK) - (A+BL)^\tp P (A+BL) . \label{eq:identity1}
\end{align}
In the case of $X=P$, identity \eqref{eq:identity1} specializes to
\begin{align}
    \cR(P) - \cR^\prime(P, P) &= 
    \begin{bmatrix}
    I \\ K
    \end{bmatrix}^\tp
    Q
    \begin{bmatrix}
    I \\ K
    \end{bmatrix}. \label{eq:identity2}
\end{align}

\section{Generic Newton methods}
First we consider finding a solution to the equation $f(x)=0$ where $f: \RR^n \to \RR^n$, whose total derivative at a point $x$ is $f^\prime(x) \in \RR^{n \times n}$.
The methods under consideration can be understood and derived as the numerical integration of the following Newton-Leibniz integral from the second fundamental theorem of calculus:
\begin{align}
    0 = f(x) = f\left(x_{k}\right)+\int_{x_{k}}^{x} f^{\prime}(t) dt. \label{eq:newton_integral}
\end{align}

\subsection{Newton method}
The Newton method, due originally in heavily modified form to \cite{newton1711analysi, raphson1702analysis} and originally in the general differential form to \cite{simpson1740essays} (see the historical notes of \cite{kollerstrom1992thomas, deuflhard2012short}), begins with an initial guess $x_0$ then proceeds with iterations
\begin{align*}
    x_{k+1} &= x_k - f^\prime(x_k)^{-1} f(x_k),
\end{align*}
until convergence. 
Intuitively, the Newton method forms a linear approximation $f(x_k) + f^\prime(x_k)(x-x_k)$ to the function $f$ at $x_k$, and assigns the point where the linear approximation crosses $0$ as the next iterate. 
The Newton method can be derived from \eqref{eq:newton_integral} by using left rectangular integration.

The Newton update can be rearranged into the Newton equation
\begin{align*}
    f^\prime(x_k) (x_{k+1} - x_k) = - f(x_k),
\end{align*}
where the left-hand side is recognized as the \emph{directional derivative} of $f$ evaluated at point $x_k$ in direction $x_{k+1} - x_k$. This rearrangement implies that the Newton method \emph{does not require explicit evaluation of the entire total derivative $f^\prime(x_k)$} so long as a suitable direction $x_{k+1} - x_k$ can be found which solves the Newton equation. This will become important in the LQR setting as we use this fact to avoid notating and computing large order-4 tensors.

This technique uses derivative information at a single point and is known to achieve quadratic convergence in a neighborhood of the root (\cite{kantorovich1948newton}).
In the setting of both continuous- and discrete-time LQR, this algorithm is known to achieve quadratic convergence globally, as shown by \cite{kleinman1968iterative, hewer1971iterative, bu2020lqr}.

\subsection{Mid-point Newton method}
The midpoint Newton method, due originally to \cite{traub1964iterative}, begins with an initial guess $x_0$ then proceeds with iterations
\begin{align*}
    x^N_{k+1} &= x_k - f^\prime(x_k)^{-1} f(x_k), \\
    x^M_{k} &= \frac{1}{2} (x_k+ x^N_{k+1}), \\
    x_{k+1} &= x_k - f^\prime( x^M_{k} )^{-1} f(x_k),
\end{align*}
until convergence. The midpoint Newton method can be derived from \eqref{eq:newton_integral} by using midpoint rectangular integration.
Intuitively, much like the Newton method, the midpoint Newton method forms a linear approximation $f(x_k) + f^\prime(x^M_k)(x-x_k)$ to the function $f$ at $x_k$, and assigns the point where the linear approximation crosses $0$ as the next iterate. The distinction is that the slope of the linear approximation is not evaluated at $x_k$ as in the Newton method, but rather at the midpoint $x^M_{k} = \frac{1}{2} (x_k + x^N_{k+1})$ where $x^N_{k+1}$ is the Newton iterate.

The updates can be rearranged into the Newton equations
\begin{align*}
    f^\prime(x_k)     (x^N_{k+1} - x_k) &= - f(x_k), \\
    f^\prime(x^M_{k}) (x_{k+1} - x_k) &= - f(x_k),
\end{align*}
where the left-hand side of the first equation is recognized as the \emph{directional derivative} of $f$ evaluated at point $x_k$ in direction $x^N_{k+1} - x_k$; the second equation is of the same form. This rearrangement implies that the midpoint Newton method \emph{does not require explicit evaluation of the entire total derivative $f^\prime(x_k)$} so long as a suitable direction $x_{k+1} - x_k$ can be found which solves the Newton equation. This will become important in the LQR setting as we use this fact to avoid notating and computing large order-4 tensors.

Each iteration in this technique uses derivative information at two points, $x_k$ and $x^M_{k}$. This method has been shown to achieve cubic convergence in a neighborhood of the root by \cite{nedzhibov2002few, homeier2004, babajee2006}.


\section{Exact midpoint policy iteration}
We now consider application of the midpoint Newton method to the Riccati equation \eqref{eq:ricc}.
Although \eqref{eq:ricc} could be brought to the vector form $f(x)=0$ by vectorization with $x=\svec(P)$ and $f(x) = \svec(\cR(\smat(x)))$, it will be simpler to leave the equations in matrix form, which is possible due to the special form of the Newton-type updates, which only involve directional derivatives (and not total derivatives).
Applying the midpoint Newton update to \eqref{eq:ricc} yields
\begin{align*}
    P^N_{k+1} &= P_k - \cR^\prime(P_k)^{-1} (\cR(P_k)), \\
    P^M_k &= \frac{1}{2} \left( P^N_{k+1} + P_k \right), \\
    P_{k+1} &= P_k - \cR^\prime(P^M_k)^{-1} (\cR(P_k)),
\end{align*}
The updates can be rearranged into the Newton equations
\begin{align}
    \cR^\prime(P_k, P^N_{k+1}-P_k) &= -\cR(P_k) \label{eq:midpoint_newton_equation1} \\
    \cR^\prime(P^M_k, P_{k+1}-P_k) &= -\cR(P_k) \label{eq:midpoint_newton_equation2}
\end{align}
and further by linearity of $\cR^\prime(\cdot, X)$ in $X$ to
\begin{align*}
    \cR^\prime(P_k, P^N_{k+1}) &= \cR^\prime(P_k, P_k) - \cR(P_k), \\
    \cR^\prime(P^M_k, P_{k+1}) &= \cR^\prime(P^M_k, P_k) - \cR(P_k).
\end{align*}
Recalling the expression for $\cR^\prime$in \eqref{eq:dR} for the left-hand sides and applying the identities in \eqref{eq:identity1} and \eqref{eq:identity2} to the right-hand sides, these become the Lyapunov equations
\begin{align*}
    P^N_{k+1} &= (A+BK_k)^\tp P^N_{k+1} (A+BK_k) + Q + K_k^\tp R K_k, \\
      P_{k+1} &= (A+BL_k)^\tp P^N_{k+1} (A+BL_k) + Q + K_k^\tp R K_k + (A+BK_k)^\tp P_{k} (A+BK_k) \\
      & \hspace{7.8cm} - (A+BL_k)^\tp P_{k} (A+BL_k),
\end{align*}
or more compactly,
\begin{align*}
    P^N_{k+1} &= \dlyap(F^N, S^N), \\
      P_{k+1} &= \dlyap(F^M, S^M),
\end{align*}
where
\begin{align*}
\begin{array}{r@{\,} c@{\,} l@{\quad} r@{\,} c@{\,} l@{\,} }
    F^N &= & A+BK_k, & S^N &= & \begin{bmatrix} I \\ K_k \end{bmatrix}^\tp  Q  \begin{bmatrix} I \\ K_k \end{bmatrix} \\
    F^M &= & A+BL_k, & S^M &= & \begin{bmatrix} I \\ K_k \end{bmatrix}^\tp  Q  \begin{bmatrix} I \\ K_k \end{bmatrix} + (A+BK_k)^\tp P_{k} (A+BK_k) \\
    \ & \ & \ & \ & \ & \hspace{2.6cm} - (A+BL_k)^\tp P_{k} (A+BL_k)
\end{array}
\end{align*}
where
\begin{align*}
    K_k = \cK(P_k), \qquad
    L_k = \cK(M_k), \qquad
    M_k = \frac{1}{2} (P_k + P^N_{k+1}) .
\end{align*}
These updates are collected in the full midpoint policy iteration in Algorithm \ref{algo:exact_mpi}.

\begin{algorithm}
\caption{Exact midpoint policy iteration (MPI)}
\begin{algorithmic}[1]
\label{algo:exact_mpi}
    \REQUIRE System matrices $A,B$, penalty matrix $Q$, initial stabilizing gain $K_0$, tolerance $\eps$
    \STATE Initialize: $k = 0$, $P_{-1} = \infty I_n$, and $P_0 = \dlyap(F, S)$ \\ where $F =  A+BK_0$ and $S = \begin{bmatrix} I & K_0^\tp \end{bmatrix}  Q  \begin{bmatrix} I & K_0^\tp \end{bmatrix}^\tp.$
    \WHILE{$\|P_{k} - P_{k-1}\| > \eps$}
    \STATE Compute 
    $K_k = \cK(P_k).$
    \STATE Compute
    $F^N =  A+BK_k$,
    and $S^N = \begin{bmatrix} I & K_k^\tp \end{bmatrix}  Q  \begin{bmatrix} I & K_k^\tp \end{bmatrix}^\tp.$
	\STATE Solve 
    $P^N_{k+1} = \dlyap(F^N, S^N)$
    \STATE Compute 
    $M_k = \frac{1}{2} (P_k + P^N_{k+1})$ and
    $L_k = \cK(M_k).$
    \STATE Compute
    $F^M = A+BL_k$, and\\
    $S^M = \begin{bmatrix} I & K_k^\tp \end{bmatrix}  Q  \begin{bmatrix} I & K_k^\tp \end{bmatrix}^\tp + (A+BK_k)^\tp P_{k} (A+BK_k) - (A+BL_k)^\tp P_{k} (A+BL_k).$
    \STATE Solve
    $P_{k+1} = \dlyap(F^M, S^M).$
    \STATE $k \leftarrow k+1$
    \ENDWHILE
\ENSURE $P_k$, $K_k = \cK(P_k)$
\end{algorithmic}
\end{algorithm}

\begin{proposition} \label{prop:exact_convergence}
Consider Exact Midpoint Policy Iteration in Algorithm \ref{algo:exact_mpi}. For any feasible problem instance, there exists a neighborhood around the optimal gain $K^*$ from which any initial gain $K_0$ yields cubic convergence, i.e.
$
    \| K_{k+1} - K^* \| \leq \mathcal{O} \left(\| K_{k} - K^* \|^3 \right)
$
where $\| \cdot \|$ is any matrix norm.
\end{proposition}
\begin{proof}
By the assumptions on $(A, B, Q)$, the closed-loop matrix under the optimal gain satisfies $\rho(A+BK^*) < 1$.
Since the spectral radius of the closed-loop matrix $A+BK$ is continuous with respect to (each entry of) the gain $K$ (see \cite{Tyrtyshnikov2012}), it follows that there exists a radius $\eps_0 > 0$ and ball $\mathcal{B}_0 = \{ K : \ \|K - K^*\| < \eps_0 \}$ around the optimal gain $K^*$ within which any gain $K$ is stabilizing.

For use later, define the following quantities and operators in terms of the system data $(A, B, Q)$ and gain $K$.
Let $P$ be the solution to
\begin{align*}
    P = {F}^\tp P F + S
\end{align*}
where $F = A + BK$ and $S = \begin{bmatrix} I & {K}^\tp \end{bmatrix}  Q  \begin{bmatrix} I & {K}^\tp \end{bmatrix}^\tp$.
Let $K^N = \cK(P)$.
Let $P^N$ be the solution to 
\begin{align*}
    P^N = {F^N}^\tp P F^N + S^N
\end{align*}
where $F^N = A+BK^N$ and $S^N = \begin{bmatrix} I & {K^N}^\tp \end{bmatrix}  Q  \begin{bmatrix} I & {K^N}^\tp \end{bmatrix}^\tp$. 
Let $P^M = \frac{1}{2}(P + P^N)$.
Let $K^M = \cK(P^M)$ and $F^M = A+BK^M$.
Define the operator
\begin{align*}
    \mathcal{Y}(K) = P^N - {F^M}^\tp P^N F^M + \cR(P).
\end{align*}

Now, at the optimal gain $K = K^* = \cR(P^*)$, we have that $K^N = K^M = K^*$ and $P^N = P^M = P^*$ where $P^*$ solves the Riccati equation $\cR(P^*) = 0$. Therefore,
\begin{align*}
    \mathcal{Y}(K^*)
    &= P^* - (A+BK^*)^\tp P^* (A+BK^*) + 0 \\
    &= \begin{bmatrix} I & {K^*}^\tp \end{bmatrix}  Q  \begin{bmatrix} I & {K^*}^\tp \end{bmatrix}^\tp \succ 0
\end{align*}
By inspection of all the preceding relevant quantities, $\mathcal{Y}(K)$ is continuous with respect to $K$, and therefore there exists a radius $\eps_1 > 0$ and ball $\mathcal{B}_1 = \{ K : \ \|K - K^*\| < \eps_1 \}$ around the optimal gain $K^*$ within which any gain $K$ satisfies $\mathcal{Y}(K) \succ 0$. Define $\eps = \min(\eps_0, \eps_1)$ and likewise $\mathcal{B} = \{ K : \ \|K - K^*\| < \eps \}$.

Consider an arbitrary matrix $P$ computed from an arbitrary gain $K$ within $\mathcal{B}$ as the solution to the Lyapunov equation
\begin{align}
    P = (A+BK)^\tp P (A+BK) + \begin{bmatrix} I & K^\tp \end{bmatrix}  Q  \begin{bmatrix} I & K^\tp \end{bmatrix}^\tp, \label{eq:generic_dlyap}
\end{align}
which is well defined since $K$ is stabilizing and $Q \succ 0$. Define the set 
\begin{align*}
    \mathcal{P} = \{P : P \text{ solves \eqref{eq:generic_dlyap} with } K \in \mathcal{B} \}.
\end{align*}
Theorem 2 of \cite{homeier2004} requires that the inverse $\cR(P)^{-1}$ exist everywhere in $\mathcal{B}$; however the proof of Theorem 2 of \cite{homeier2004} only uses this assumption in order to ensure that solutions to the Newton equations \eqref{eq:midpoint_newton_equation1} and \eqref{eq:midpoint_newton_equation2} exist and are unique. Therefore, it suffices to prove just that solutions to the Newton equations \eqref{eq:midpoint_newton_equation1} and \eqref{eq:midpoint_newton_equation2} exist and are unique for any $P$ in $\mathcal{P}$.
Write the first Newton equation \eqref{eq:midpoint_newton_equation1} as 
\begin{align*}
    \cR(P, P^N - P) = -\cR(P).
\end{align*}
Using the expression \eqref{eq:dR}, this can be rewritten as
\begin{align}
    P - P^N = (A+BK^N)^\tp (P - P^N)(A+BK^N) -\cR(P). \label{eq:newton_equation1_rewrite}
\end{align}
where $K^N = \cK(P)$. 
By e.g. \cite{bertsekas1995dynamic} the matrix $-\cR(P) \succ 0$ (this is related to convergence of value iteration). 
Also, because $K^N = \cK(P)$ we may apply the Wonham-like identity developed in \cite{hewer1971iterative}
\begin{align*}
    P = {F^N}^\tp P F^N + S
\end{align*}
where $F^N = A+BK^N$ and 
\begin{align*}
    S = \begin{bmatrix} I & {K^N}^\tp \end{bmatrix}  Q  \begin{bmatrix} I & {K^N}^\tp \end{bmatrix}^\tp + (K - K^N)^\tp (Q_{uu} + B^\tp P B) (K - K^N) \succ 0 ,
\end{align*}
which shows that the gain $K^N$ is stabilizing i.e. $A+BK^N$ is Schur stable. 
Therefore the solution to \eqref{eq:newton_equation1_rewrite} is unique and well defined (and positive definite).

Similarly, write the second Newton equation \eqref{eq:midpoint_newton_equation2} as 
\begin{align*}
    \cR(P^M, P^+ - P) = -\cR(P)
\end{align*}
where $P^M = \frac{1}{2}(P + P^N)$ and $P^N$ solves the first Newton equation \eqref{eq:newton_equation1_rewrite}, equivalently
\begin{align*}
    P^N &= {F^N}^\tp P^N F^N + S^N
\end{align*}
where $S^N = \begin{bmatrix} I & {K^N}^\tp \end{bmatrix}  Q  \begin{bmatrix} I & {K^N}^\tp \end{bmatrix}^\tp$.
Using the expression \eqref{eq:dR}, the second Newton equation \eqref{eq:midpoint_newton_equation2} can be further rewritten as
\begin{align}
    P - P^+ = (A+BK^M)^\tp (P - P^+)(A+BK^M) -\cR(P). \label{eq:newton_equation2_rewrite}
\end{align}
where $K^M = \cK(P^M)$. 
Again, by e.g. \cite{bertsekas1995dynamic} the matrix $-\cR(P) \succ 0$.
By construction $P^M = \frac{1}{2}(P + P^N)$, so combining the expressions
\begin{align*}
    P &= {F^N}^\tp P F^N + S, \\
    P^N &= {F^N}^\tp P^N F^N + S^N 
\end{align*}
we obtain
\begin{align*}
    P^M 
    &= \frac{1}{2}(P + P^N) \\
    &= \frac{1}{2}\left( {F^N}^\tp P F^N + S + {F^N}^\tp P^N F^N + S^N \right) \\
    &= {F^N}^\tp \left( \frac{1}{2} (P+P^N) \right) F^N + \frac{1}{2}(S + S^N) \\
    &= {F^N}^\tp P^M F^N + \frac{1}{2}(S + S^N)
\end{align*}
and using $K^M = \cK(P^M)$ and the Wonham-like identity of \cite{hewer1971iterative} again we obtain
\begin{align*}
    P^M = {F^M}^\tp P^M F^M + S^M
\end{align*}
where $F^M = A + B K^M$ and
\begin{align*}
    S^M = \frac{1}{2}(S + S^N) + (K^N - K^M)^\tp (Q_{uu} + B^\tp P^M B) (K^N - K^M) \succ 0,
\end{align*}
which shows that the gain $K^M$ is stabilizing i.e. $A + BK^M$ is Schur stable.
Therefore the solution to \eqref{eq:newton_equation2_rewrite} is unique and well defined (and positive definite).
Since $P$ in $\mathcal{P}$ was arbitrary, we have proved the assertion that solutions to the Newton equations \eqref{eq:midpoint_newton_equation1} and \eqref{eq:midpoint_newton_equation2} exist and are unique for any $P$ in $\mathcal{P}$. Furthermore, $P^+$ equivalently solves the Lyapunov equation
\begin{align*}
    P^+ = {F^M}^\tp P^+ F^M + S^+
\end{align*}
where
\begin{align*}
    S^+ = P^N - {F^M}^\tp P^N F^M + \cR(P) \succ 0
\end{align*}
where the positive definiteness of $S^+$ follows by the restriction of $K$ to $\mathcal{B}_1$, and therefore $P^+$ proves stability of $K^M$.

Theorem 2 of \cite{homeier2004} also requires that $\cR$ be sufficiently smooth with bounded derivatives up to third order in $\mathcal{B}$. Any $P \in \mathcal{P}$ is positive definite and bounded above since $P$ solves \eqref{eq:generic_dlyap} and $K \in \mathcal{K}$ is stable. Also, by assumption we have $Q_{uu} \succ 0$. Therefore, the term $Q_{uu} + B^\tp P B \succeq Q_{uu} \succ 0$ so its inverse is well defined and bounded above. By examination of \eqref{eq:ricc} and \eqref{eq:dR} it is evident that $\cR$ and $\cR^\prime$ are analytic functions with upper bounds by the preceding arguments. Higher-order derivatives of $\cR$ follow similar Lyapunov equations as $\cR^\prime$ and are thus also upper bounded on $\mathcal{P}$. 

The assumptions of Theorem 2 of \cite{homeier2004} are satisfied, and thus we conclude that the iterates converge $\|P_{k+1} - P^*\| < c \|P_{k} - P^*\|$ for some $c \in [0, 1)$ and do so at a cubic rate $\| P_{k+1} - P^* \| \leq \mathcal{O} \left(\| P_{k} - P^* \|^3 \right)$.
Consider $K_{k} = K$ so $P_{k+1} = P^+$, and $L_k = K^M$.
Since $K_{k+1} = \cK(P_{k+1})$ and $P_{k+1}$ proves stability of a gain matrix $K^M$, this implies Schur stability of every $A+BK_k$ (using the Wonham-like identity of \cite{hewer1971iterative}). Likewise, since the sequence of $P_k$ approach the limit $P^*$ and $K_k = \cK(P_k)$, the sequence of $K_k$ approach the limit $K^*$.
Since $\|K_k - K^*\| = \| \cK(P_k) - \cK(P^*) \| = \mathcal{O}(\|P_{k} - P^*\|)$, we conclude the same cubic convergence result holds for $K_k$.
\end{proof}

\section{Approximate midpoint policy iteration}
In the model-free setting we do not have access to the dynamics matrices $(A,B)$, so we cannot execute the updates in Algorithm \ref{algo:exact_mpi}.
However, the gain $K=\cK(P)$ can be computed solely from the state-action value matrix $H=\cH(P)$ as $K=-H_{uu}^{-1}H_{ux}$. Thus, if we can obtain accurate estimates of $H$, we can use the estimate of $H$ to compute $K$ and we need not perform any other updates that depend explicitly on $(A,B)$. We begin by summarizing an existing method in the literature for estimating state-action value functions from observed state-and-input trajectories.

\subsection{State-action value estimation}
First, we connect the matrix $H$ with the (relative) state-action value ($\cQ$) function, which determines the (relative) cost of starting in state $x = x_0$, taking action $u = u_0$, then following the policy $u_t=K x_t$ thereafter:
\begin{align*}
    \Tr(PW) +
    \cQ_{K}(x,u)  
    &= 
    \begin{bmatrix}
    x \\
    u
    \end{bmatrix}^\tp
    \begin{bmatrix}
    Q_{xx} & Q_{xu}\\
    Q_{ux} & Q_{uu}
    \end{bmatrix}
    \begin{bmatrix}
    x \\
    u
    \end{bmatrix}
    + \underset{w}{\EE} \Big[ (Ax+Bu+w)^\tp P (Ax+Bu+w) \Big] \nonumber \\
    &=
    \begin{bmatrix}
    x \\
    u
    \end{bmatrix}^\tp
    \begin{bmatrix}
    H_{xx} & H_{xu}\\
    H_{ux} & H_{uu}
    \end{bmatrix}
    \begin{bmatrix}
    x \\
    u
    \end{bmatrix} 
    + \Tr(PW)
\end{align*}
or simply
\begin{align*}
    \cQ_{K}(x,u)  
    &= 
    \begin{bmatrix}
    x \\
    u
    \end{bmatrix}^\tp
    \begin{bmatrix}
    H_{xx} & H_{xu}\\
    H_{ux} & H_{uu}
    \end{bmatrix}
    \begin{bmatrix}
    x \\
    u
    \end{bmatrix} 
\end{align*}
where $H=\cH(P)$ and $P$ is the solution to
\begin{align}
    \dlyap\left(A+BK, \begin{bmatrix} I \\ K \end{bmatrix}^\tp \begin{bmatrix} Q_{xx} & Q_{xu} \\ Q_{ux} & Q_{uu} \end{bmatrix} \begin{bmatrix} I \\ K \end{bmatrix}\right) \label{eq:generic_save_P}
\end{align}
From this expression it is clear that a state-input trajectory, or ``rollout,'' $\cD = \{ x_t, u_t \}_{t=0}^{\ell}$ must satisfy this cost relationship, which can be used to estimate $H$. In particular, least-squares temporal difference learning for $\cQ$-functions (LSTDQ) was originally introduced by \cite{lagoudakis2003} and analyzed by \cite{abbasi2019model, krauth2019finite}, and is known to be a consistent and unbiased estimator of $H$. 
Following the development of \cite{krauth2019finite}, the LSTDQ estimator is summarized in Algorithm \ref{algo:lstdq}.
\begin{algorithm}
\caption{\lstdq: Least-squares temporal difference learning for $Q$-functions}
\begin{algorithmic}[1]
\label{algo:lstdq}
    \REQUIRE Rollout $\cD = \{ x_t, u_t \}_{t=0}^{\ell}$, gain matrix $K^{\text{eval}}$, penalty matrix $Q$, noise covariance $W$.
    \STATE Compute augmented rollout $\{z_t, v_t, c_t \}_{t=0}^{\ell}$ where
    $
        z_t = \begin{bmatrix} x \\ u  \end{bmatrix}, \
        v_t = \begin{bmatrix} x \\ K^{\text{eval}} x \end{bmatrix}, \
        c_t = z_t^\tp Q z_t.
    $
    \STATE Use feature map $\phi(z) = \svec \left(z z^\tp \right)$ and noise quantity 
    $\psi = \svec \left( \begin{bmatrix} I \\ K^{\text{eval}} \end{bmatrix} W \begin{bmatrix} I \\ K^{\text{eval}} \end{bmatrix}^\tp \right)$ 
    and compute the parameter estimate \\
    {\begin{center} $ \hat{\Theta} = \Big( \sum_{t=1}^\ell \phi(z_t) (\phi(z_t) - \phi(v_{t+1}) + \psi)^\tp \Big)^\dag \sum_{t=1}^\ell \phi(z_t) c_t. $ \end{center}}
\ENSURE $\hat{H} = \smat(\hat{\Theta})$.
\end{algorithmic}
\end{algorithm}

We collect rollouts to feed into Algorithm \ref{algo:lstdq} via Algorithm \ref{algo:rollout}, i.e. by initializing the state with $x_0$ drawn from the given initial state distribution $\mathcal{X}_0$, then generating control inputs according to $u_t = K^{\text{play}} x_t + u^{\text{explore}}_t$ where $K^{\text{play}}$ is a stabilizing gain matrix, and $u^{\text{explore}}_t$ is an exploration noise drawn from a distribution $\mathcal{U}_t$, assumed Gaussian in this work, to ensure persistence of excitation.
\begin{algorithm}
\caption{$\rollout$: Rollout collection}
\begin{algorithmic}[1]
\label{algo:rollout}
    \REQUIRE Gain $K^{\text{play}}$, rollout length $\ell$, initial state distribution $\cX_0$, exploration distributions $\{ \cU_t \}_{t=0}^\ell$.
    \STATE Initialize state $x_0 \sim \mathcal{X}_0$
    \FOR{$t=0,1,2,\ldots,\ell$}
    \STATE Sample exploratory control input $u^{\text{explore}}_t \sim \mathcal{U}_t$ and disturbance $w_t \sim W$
    \STATE Generate control input $u_t = K^{\text{play}} x_t + u^{\text{explore}}_t$
    \STATE Record state $x_t$ and input $u_t$
    \STATE Update state according to $x_{t+1} = A x_t +  B u_t + w_t$
    \ENDFOR
\ENSURE $\cD = \{ x_t, u_t \}_{t=0}^{\ell}$.
\end{algorithmic}
\end{algorithm}

Note that LSTDQ is an off-policy method, and thus the gain $K^{\text{play}}$ used to generate the data in Algorithm \ref{algo:rollout} and the gain $K^{\text{eval}}$ whose state-action value matrix is estimated in Algorithm \ref{algo:lstdq} need not be identical. We will use this fact in the next section to give an off-policy, offline (OFF) and on-policy, online (ON) version of our algorithm.
Likewise, the penalty matrix $Q$ used in Algorithm \ref{algo:lstdq} need not be the same as the one in the original problem statement, which is critical to developing the model-free midpoint update in the next section.

\subsection{Derivation of approximate midpoint policy iteration}
We have shown that estimates of the state-action value matrix $H$ can be obtained by $\lstdq$ using either off-policy or on-policy data. In the following development, (OFF) denotes a variant where a single off-policy rollout $\cD$ is collected offline before running the system, and (ON) denotes a variant where new on-policy rollouts are collected at each iteration. Also, an overhat symbol `` $\hat{}$ '' denotes an estimated quantity while the absence of one denotes an exact quantity.

In approximate policy iteration, we can simply form the estimate $\hat{H}_k$ using $\lstdq$ (see \cite{krauth2019finite}).
For approximate midpoint policy iteration, the form of $\hat{H}_k$ is more complicated and requires multiple steps.
To derive approximate midpoint policy iteration, we will re-order some of the steps in the loop of Algorithm \ref{algo:exact_mpi}. Specifically, move the gain calculation in step 3 to the end after step 9. 
We will also replace explicit computation of the value function matrices with estimation of state-action value matrices, i.e. subsume the pairs of steps 4, 5 and 8,9 into single steps, and work with $H$ instead of $P$.
Thus, at the beginning of each iteration we have in hand an estimated state-action value matrix $\hat{H}_k$ and gain matrix $\hat{K}_k$ satisfying $\hat{K}_k = -\hat{H}_{uu,k}^{-1} \hat{H}_{ux,k}$.

First we translate steps 4, 5, 6, and 7 to a model-free version.
Working backwards starting with step 7, in order to estimate $L_k$, it suffices to estimate $\cH(M_k)$ since $L_k = -\cH(M_k)_{uu}^{-1} \cH(M_k)_{ux}$.
In order to find $\cH(M_k)$, notice that the operator $\cH(X)$ is linear in $X$, so
\begin{align*}
    \cH(M_k) = \cH\left( \frac{1}{2} (P_k + P^N_{k+1}) \right) = \frac{1}{2} \left( \cH(P_k) + \cH(P^N_{k+1}) \right) .
\end{align*}
Therefore we can estimate $\cH(M_k)$ by estimating $\cH(P_k)$ and $\cH(P^N_{k+1})$ separately and taking their midpoint. Since the estimate $\hat{H}_k$ of $\cH(P_k)$ is known from the prior iteration, what remains is to find an estimate $\hat{H}^N_{k+1}$ of $\cH(P^N_{k+1})$ by
\begin{align*}
    \text{collecting } \cD^N &= \rollout(\hat{K}_k, \ell, \cX_0, \{ \cU_t \}_{t=0}^\ell) \tag{ON}, \\
    & \text{ or } \\
    \text{using } \cD^N &= \cD \tag{OFF},
\end{align*}
and estimating $\hat{H}^N_{k+1} = \lstdq(\cD^N, \hat{K}_k, Q)$. \\
Then we form the estimated gain $\hat{L}_k = -{\hat{H}^M_{uu,k}}{}^{-1} \hat{H}^M_{ux,k}$ where $\hat{H}^M_k = \frac{1}{2} (\hat{H}_k + \hat{H}^N_{k+1})$.

Now we translate steps 8, 9, and 2 to a model-free version. Working backwards, starting with step 2, in order to estimate $K_{k+1}$, it suffices to find an estimate $\hat{H}_{k+1}$ of matrix $\cH(P_{k+1})$ since $K_{k+1} = -\cH(P_{k+1})_{uu}^{-1} \cH(P_{k+1})_{ux}$.
From steps 8 and 9, we want to estimate
\begin{align}
    H_{k+1} &= \cH(P_{k+1})
    =
    Q
    + 
    \begin{bmatrix}
    A & B
    \end{bmatrix}^\tp
    P_{k+1}
    \begin{bmatrix}
    A & B
    \end{bmatrix}, \label{eq:ampi_H} \\
    \text{where } \quad P_{k+1}
    &=
    \dlyap\left(
    F^M,
    S^M
    \right), \label{eq:ampi_P} \\
    F^M &= A+BL_k, \nonumber \\
    S^M
    &=
    \begin{bmatrix}
    I \\ K_k
    \end{bmatrix}^\tp
    \left(
    Q +
    \begin{bmatrix}
    A & B
    \end{bmatrix}^\tp
    P_k
    \begin{bmatrix}
    A & B
    \end{bmatrix}
    \right)
    \begin{bmatrix}
    I \\ K_k
    \end{bmatrix}
    -
    \begin{bmatrix}
    I \\ L_k
    \end{bmatrix}^\tp
    \begin{bmatrix}
    A & B
    \end{bmatrix}^\tp
    P_k
    \begin{bmatrix}
    A & B
    \end{bmatrix}
    \begin{bmatrix}
    I \\ L_k
    \end{bmatrix}. \nonumber
\end{align}
Comparing the two arguments to $\dlyap(\cdot, \cdot)$ in \eqref{eq:generic_save_P} and \eqref{eq:ampi_P}, we desire both
\begin{align}
    A+BK &= A+BL_k, \label{eq:compatability1} \\
    \begin{bmatrix}
    I \\ K
    \end{bmatrix}^\tp
    Q^M
    \begin{bmatrix}
    I \\ K
    \end{bmatrix}
    &=
    S^M. \label{eq:compatability2}
\end{align}
Clearly it suffices to take $K=L_k$ in \eqref{eq:compatability1}.
Notice that, critically, all quantities in $S^M$ on the right-hand side of \eqref{eq:compatability1} have been estimated already, i.e. $\hat{K}_k$, $\hat{L}_k$, $\hat{H}_k$ have been calculated already and
\begin{align*}
    Q +
    \begin{bmatrix}
    A & B
    \end{bmatrix}^\tp
    P_k
    \begin{bmatrix}
    A & B
    \end{bmatrix}
    &=
    H_{k}, \\
    \begin{bmatrix}
    A & B
    \end{bmatrix}^\tp
    P_k
    \begin{bmatrix}
    A & B
    \end{bmatrix}
    &=
    H_{k} - Q.
\end{align*}
Substituting $K=L_k$ in \eqref{eq:compatability1} and comparing coefficients, it suffices to estimate $Q^M$ by
\begin{align}
    \hat{Q}^M
    =
    \begin{bmatrix}
        \begin{bmatrix}
        I & \hat{K}_k^\tp
        \end{bmatrix}
        \hat{H}_k
        \begin{bmatrix}
        I & \hat{K}_k^\tp
        \end{bmatrix}^\tp
        & 0 \\
        0 & 0
    \end{bmatrix}
    -
    (\hat{H}_k - Q) . \label{eq:Qm}
\end{align}
At this point, establish the rollout $\cD^M$ either by 
\begin{align*}
    \text{collecting }  \cD^M &= \rollout(\hat{L}_k, \ell, \cX_0, \{ \cU_t \}_{t=0}^\ell), \tag{ON} \\
    & \text{or} \\
    \text{using }  \cD^M &= \cD. \tag{OFF}
\end{align*}
Then the matrix $\hat{H}^O_{k+1} = \lstdq(\cD^M, \hat{L}_k, \hat{Q}^M)$ estimates
\begin{align*}
    H^O_{k+1} 
    =
    Q^M 
    + 
    \begin{bmatrix}
    A & B
    \end{bmatrix}^\tp
    P_{k+1}
    \begin{bmatrix}
    A & B
    \end{bmatrix} .
\end{align*}
However, we need
\begin{align*}
    H_{k+1}
    =
    Q
    + 
    \begin{bmatrix}
    A & B
    \end{bmatrix}^\tp
    P_{k+1}
    \begin{bmatrix}
    A & B
    \end{bmatrix} ,
\end{align*}
which is easily found by offsetting $H^O_{k+1}$ as
\begin{align*}
    H_{k+1} = H^O_{k+1} + (Q - Q^M),
\end{align*}
and thus
\begin{align*}
    \hat{H}_{k+1} = \hat{H}^O_{k+1} + (Q - \hat{Q}^M)
\end{align*}
estimates $H_{k+1}$.
One further consideration to address is the initial estimate $\hat{H}_0$; since we do not have a prior iterate to use, we simply collect $\cD = \rollout(\hat{K}_0, \ell, \cX_0, \{ \cU_t \}_{t=0}^\ell)$ and estimate $\hat{H}_0 = \lstdq(\cD, \hat{K}_0, Q)$ i.e. the first iteration will be a standard approximate policy iteration/Newton step. Importantly, the initial gain $\hat{K}_0$ must stabilize the system so that the value functions are finite-valued.
Also, although a convergence criterion such as $\|\hat{H}_{k} - \hat{H}_{k-1}\| > \eps$ could be used, it is more straightforward to use a fixed number of iterations $N$ so that the influence of stochastic errors in $\hat{H}_k$ does not lead to premature termination of the program. 
Likewise, a schedule of increasing rollout lengths $\ell$ could be used for the (ON) variant to achieve increasing accuracy, but finding a meaningful schedule which properly matches the fast convergence rate of the algorithm requires more extensive analysis.
The full set of updates are compiled in Algorithm \ref{algo:approximate_mpi}.

\begin{algorithm}
\caption{Approximate midpoint policy iteration (AMPI)}
\begin{algorithmic}[1]
\label{algo:approximate_mpi}
    \REQUIRE Penalty $Q$, gain $\hat{K}_0$, number of iterations $N$, rollout length $\ell$, distributions $\cX_0$, $\{ \cU_t \}_{t=0}^\ell$.
    \STATE Initialize: $\hat{H}_{-1} = \infty I_{n+m}$ and $k = 0$ 
    \STATE Collect $\cD = \rollout(\hat{K}_0, \ell, \cX_0, \{ \cU_t \}_{t=0}^\ell)$
    \STATE Estimate value matrix $\hat{H}_0 = \lstdq(\cD, K_0, Q)$.
    \WHILE{$k < N$}
    \STATE Set $\cD^N =\cD$ (OFF), or collect $\cD^N = \rollout(\hat{K}_k, \ell, \cX_0, \{ \cU_t \}_{t=0}^\ell)$ (ON)
    \STATE Estimate value matrix $\hat{H}^N_{k+1} = \lstdq(\cD^N, \hat{K}_k, Q)$.
    \STATE Form the midpoint value estimate
    $\hat{H}^M_k = \frac{1}{2} (\hat{H}_k + \hat{H}^N_{k+1}).$
    \STATE Compute the midpoint gain
    $\hat{L}_k = -{\hat{H}^M_{uu,k}}{}^{-1} \hat{H}^M_{ux,k}.$
    \STATE Set $\cD^M =\cD$ (OFF), or collect $\cD^M = \rollout(\hat{L}_k, \ell, \cX_0, \{ \cU_t \}_{t=0}^\ell)$ (ON)
    \STATE Estimate $\hat{H}^O_{k+1} = \lstdq(\cD^M, \hat{L}_k, \hat{Q}^M)$ where
    $
        Q^M 
        =
        \begin{bmatrix}
            \begin{bmatrix}
            I \\ \hat{K}_k
            \end{bmatrix}^\tp
            \hat{H}_k
            \begin{bmatrix}
            I \\ \hat{K}_k
            \end{bmatrix}
            & 0 \\
            0 & 0
        \end{bmatrix}
        -
        (\hat{H}_k -Q).
    $
    \STATE Compute the estimated value matrix 
    $\hat{H}_{k+1} = \hat{H}^O_{k+1} + ( Q - \hat{Q}^M )$.
    \STATE Compute the gain
    $\hat{K}_{k+1} = -\hat{H}_{uu,k+1}^{-1} \hat{H}_{ux,k+1}.$
    \STATE $k \leftarrow k+1$
    \ENDWHILE
\ENSURE $\hat{H}_k$, $\hat{K}_k$
\end{algorithmic}
\end{algorithm}

\begin{proposition} \label{prop:approx_convergence}
Consider Approximate Midpoint Policy Iteration in Algorithm \ref{algo:approximate_mpi}. As the rollout length $\ell$ grows to infinity, the state-action value matrix estimate $\hat H_k$ converges to the exact value. Thus, in the infinite data limit, for any feasible problem instance, there exists a neighborhood around the optimal gain $K^*$ from which any initial gain $\hat K_0$ converges cubically to $K^*$.
\end{proposition}
\begin{proof}
The claim follows by Proposition \ref{prop:exact_convergence} and the fact that $\lstdq$ is a consistent estimator \cite{lagoudakis2003, krauth2019finite}, i.e. as $\ell \to \infty$ the estimates $\hat{H}$ used in Algorithm \ref{algo:approximate_mpi} approach the true values $H$ indirectly used in Algorithm \ref{algo:exact_mpi}.
\end{proof}

\section{Numerical experiments}
In this section we compare the empirical performance of proposed midpoint policy iteration (MPI) with standard policy iteration (PI), as well as their approximate versions (AMPI) and (API).
In all experiments, regardless of whether the exact or approximate algorithm is used, we evaluated the value matrix $P_k$ associated to the policy gains $K_k$ at each iteration $k$ on the true system, i.e. the solution to 
$P_k = \dlyap\left(A+BK_k, \begin{bmatrix} I & K_k^\tp \end{bmatrix} Q \begin{bmatrix} I & K_k^\tp \end{bmatrix}^\tp\right)$. 
We then normalized the deviation $\|P_k - P^*\|$, where $P^*$ solves the Riccati equation \eqref{eq:ricc}, by the quantity $\| P^* \|$. This gives a meaningful metric to compare different suboptimal gains.
We also elected to focus on the off-policy version (OFF) of AMPI and API in order to achieve a more direct and fair comparison between the midpoint and standard methods; each is given access to precisely the same sample data and initial policy, so differences in convergence are entirely due to the algorithms. 
Nevertheless, similar results were observed in the on-policy online setting (ON), albeit with more variation between Monte Carlo runs due to differing sample data.
Python code which implements the proposed algorithms and reproduces the experimental results is available at \url{https://github.com/TSummersLab/midpoint-policy-iteration}.

\subsection{Representative example}
Here we consider one of the simplest tasks in the control discipline: regulating an inertial mass using a force input.
The stochastic continuous-time dynamics of the second-order system are
\begin{align*}
    dx = A_c x \ dt  + B_c u \ dt + dw
\end{align*}
where
\begin{align*}
    A_c = 
    \begin{bmatrix}
    0 & 1 \\
    0 & 0
    \end{bmatrix}, \quad
    B_c = 
    \begin{bmatrix}
    0 \\
    \frac{1}{\mu}
    \end{bmatrix},
\end{align*}
with mass $\mu > 0$, state $x \in \RR^2$ where the first state is the position and the second state is the velocity, force input $u \in \RR$, and $dw \in \RR^2$ is a Wiener process with covariance $W_c \succeq 0$. 
Forward-Euler discretization of the continuous-time dynamics with sampling time $\Delta t$ yields the discrete-time dynamics
\begin{align*}
    x_{t+1} = A x_t + B u_t + w_t
\end{align*}
where
\begin{align*}
    A = \begin{bmatrix}
    1 & \Delta t \\
    0 & 1
    \end{bmatrix}, \quad
    B = 
    \begin{bmatrix}
    0 \\
    \frac{\Delta t}{\mu} 
    \end{bmatrix},
\end{align*}
with $w_t \sim \mathcal{N}(0, W)$ with $W = \Delta t \cdot W_c $. 
We used $\mu = 1$, $\Delta t = 0.01$, $W_c = 0.01 I_2$, $Q = I_3$. The initial gain was chosen by perturbing the optimal gain $K^*$ in a random direction such that the initial relative error $\|P_k - P^*\|/\| P^* \| = 10$; in particular the initial gain was $K_0 = \begin{bmatrix} -0.035 & -2.087 \end{bmatrix}$.
For the approximate algorithms, we used the hyperparameters $\ell = 300$, $\mathcal{X}_0 = \mathcal{N}(0, I_2)$, $\mathcal{U}_t = \mathcal{N}(0, I_2)$ for $t=0,1,\ldots, \ell$.

The results of applying midpoint policy iteration and the standard policy iteration are plotted in Figure \ref{fig:inertial_mass_convergence}. Clearly MPI and AMPI converge more quickly to the (approximate) optimal policy than PI and API, with MPI converging to machine precision in 7 iterations vs 9 iterations for PI, and AMPI converging to noise precision in 6 iterations vs 8 iterations for API.

\begin{figure}[htb]
\floatconts
  {fig:inertial_mass_convergence}
  {\caption{Relative value error $\|P_k - P^*\|/\| P^* \|$ vs iteration count $k$ using PI and MPI on the inertial mass control problem.}}
  {\includegraphics[width=0.7\linewidth]{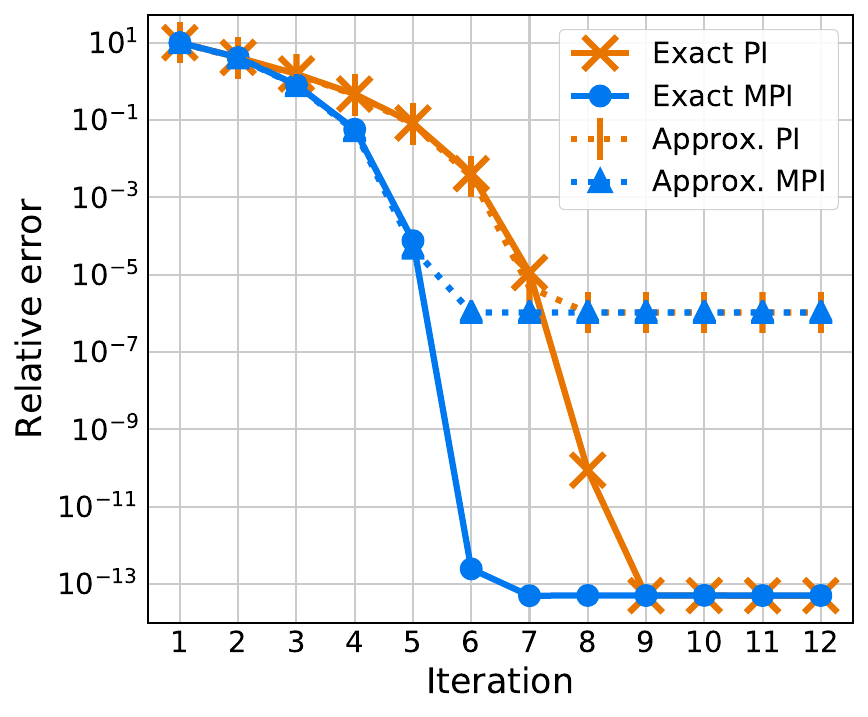}}
\end{figure}

\subsection{Randomized examples}
Next we apply the exact and approximate PI algorithms on $10000$ unique problem instances in a Monte Carlo-style approach, where problem data was generated randomly with $n=4$, $m=2$, entries of $A$ drawn from $\cN(0,1)$ and $A$ scaled so $\rho(A) \sim \text{Unif}([0, 2])$, entries of $B$ drawn from $\text{Unif}([0, 1])$, and $Q = U \Lambda U^\tp \succ 0$ with $\Lambda$ diagonal with entries drawn from $\text{Unif}([0, 1])$ and $U$ orthogonal by taking the QR-factorization of a square matrix with entries drawn from $\cN(0,1)$, where we denote the uniform distribution on the interval $[a,b]$ by $\text{Unif}([a, b])$ and the multivariate Gaussian distribution with mean $\mu$ and variance $\Sigma$ by $\mathcal{N}(\mu, \Sigma)$.
We used a small process noise covariance of $W = 10^{-6} I_4$ to avoid unstable iterates due to excessive data-based approximation error of $H$, over all problem instances.
All initial gains $K_0$ were chosen by perturbing the optimal gain $K^*$ in a random direction such that the initial relative error $\|P_k - P^*\|/\| P^* \| = 10$.
For the approximate algorithms, we used the hyperparameters $\ell = 100$, $\mathcal{X}_0 = \mathcal{N}(0, I_2)$, $\mathcal{U}_t = \mathcal{N}(0, I_2)$ for $t=0,1,\ldots, \ell$.

In Figures \ref{fig:plot_error_monte_carlo_true_scatter}, \ref{fig:plot_error_monte_carlo_false_scatter}, \ref{fig:plot_error_monte_carlo_false_scatter_on} we plot the relative value error $\|P_k - P^*\|/\| P^* \|$ over iterations, and each scatter point represents a unique Monte Carlo sample, i.e. a unique problem instance, initial gain, and rollout. 
Each plot shows the empirical distribution of errors at the iteration count $k$ labeled in the subplot titles above each plot. The x-axis is the spectral radius of $A$ which characterizes open-loop stability.
\begin{itemize}
    \item Figure \ref{fig:plot_error_monte_carlo_true_scatter} shows the results of the exact algorithms i.e. Algorithm \ref{algo:exact_mpi}.
    \item Figure \ref{fig:plot_error_monte_carlo_false_scatter} shows the results of the offline approximate algorithms i.e. Algorithm \ref{algo:approximate_mpi} (OFF).
    \item Figure \ref{fig:plot_error_monte_carlo_false_scatter_on} shows the results of the online approximate algorithms i.e. Algorithm \ref{algo:approximate_mpi} (ON).
\end{itemize}
In sub-Figures \ref{fig:plot_error_monte_carlo_true_scatter} (b), \ref{fig:plot_error_monte_carlo_false_scatter} (b), \ref{fig:plot_error_monte_carlo_false_scatter_on} (b), scatter points lying below 1.0 on the y-axis indicate that the midpoint method achieves lower error than the standard method on the same problem instance.

From Figure \ref{fig:plot_error_monte_carlo_true_scatter} (a), it is clear that MPI achieves extremely fast convergence to the optimal gain, with the relative error being less than $10^{-13}$, essentially machine precision, on almost all problem instances after just 5 iterations. From Figure \ref{fig:plot_error_monte_carlo_true_scatter} (b), we see that MPI achieves significantly lower error than PI on iteration counts $2,3,4,5$ for almost all problem instances. The relative differences in error on iteration counts $6,7$ are due to machine precision error and are negligible for the purposes of comparison i.e. after 6 iterations both algorithms have effectively converged to the same solution.

We observe very similar results using the approximate algorithms.
From Figure \ref{fig:plot_error_monte_carlo_false_scatter} (a), it is clear that AMPI achieves extremely fast convergence to a good approximation of the optimal gain, with the relative error being less than $10^{-6}$ on almost all problem instances after just 4 iterations. From Figure \ref{fig:plot_error_monte_carlo_false_scatter} (b), we see that AMPI achieves significantly lower error than API on iteration counts $2,3,4,5$ for almost all problem instances; recall that Algorithm \ref{algo:approximate_mpi} takes a standard PI step on the first iteration, explaining the identical performance on $k=1$.
Similar trends are observed in Figure \ref{fig:plot_error_monte_carlo_false_scatter_on} with the online variant (ON), but the variation is much greater. Nevertheless, AMPI provides a clear advantage on iteration counts $2,3,4,5$, beating API in terms of relative error most of the time.

\begin{figure}[!htbp]
\floatconts
  {fig:plot_error_monte_carlo_true_scatter}
  {\caption{(a) Relative value error $\|P_k - P^*\|/\| P^* \|$ using MPI and (b) ratio of relative error using MPI divided by that using PI. }}
  {%
    \subfigure[Relative error using MPI]{\includegraphics[width=0.99\linewidth]{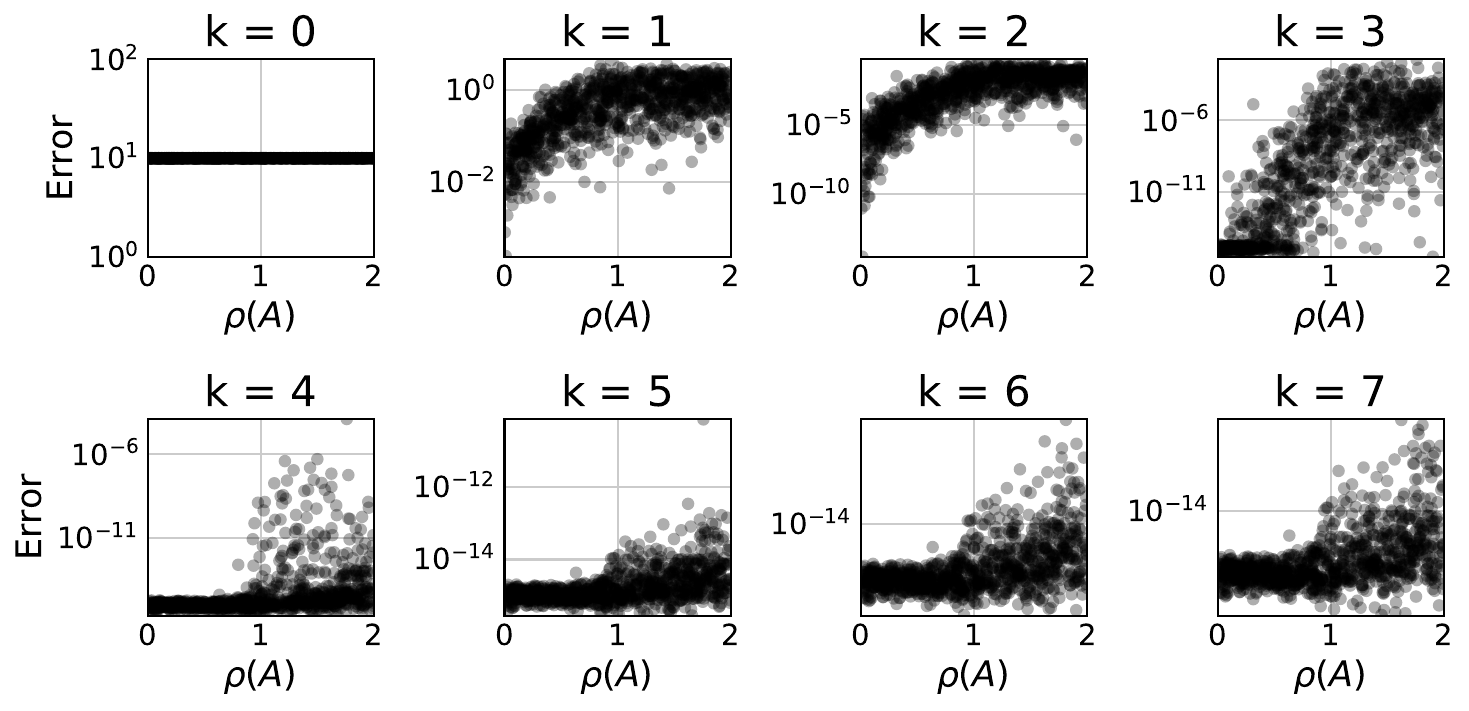}}\\ \ \\%
    \subfigure[Ratio of relative errors using MPI/PI]{\includegraphics[width=0.99\linewidth]{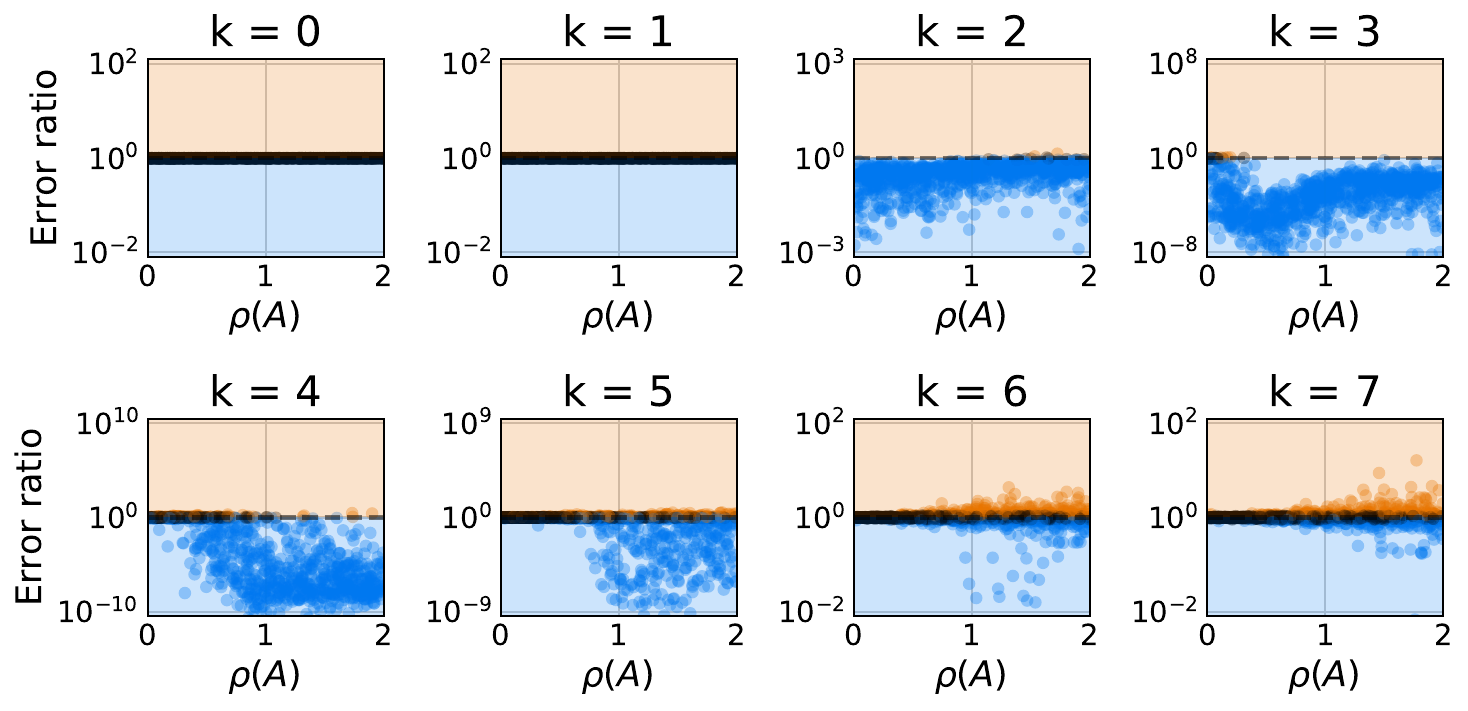}}
  }
\end{figure}

\begin{figure}[!htbp]
\floatconts
  {fig:plot_error_monte_carlo_false_scatter}
  {\caption{(a) Relative value error $\|P_k - P^*\|/\| P^* \|$ using AMPI and (b) ratio of relative error using AMPI divided by that using API, all with the offline algorithm variant (OFF).}}
  {%
    \subfigure[Relative error using AMPI (OFF)]{\includegraphics[width=0.99\linewidth]{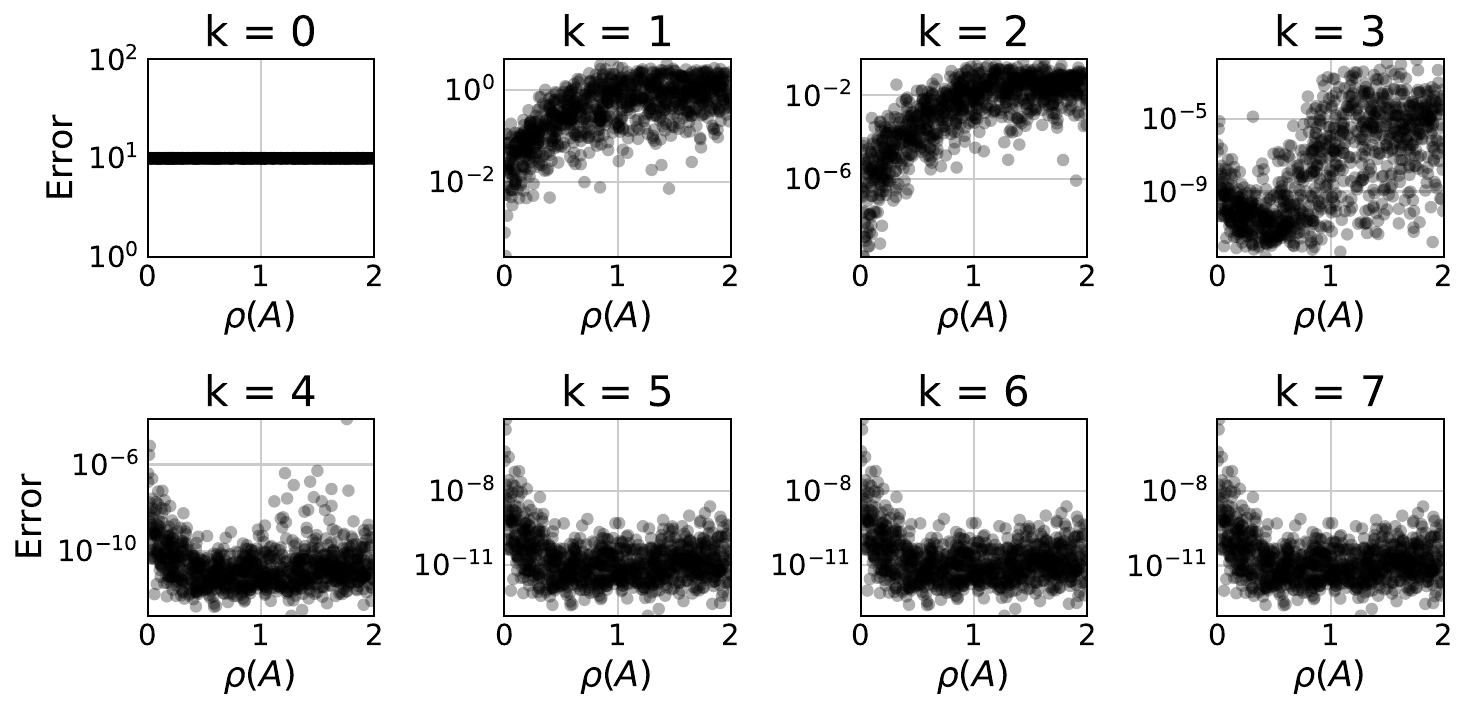}}\\ \ \\%
    \subfigure[Ratio of relative errors using AMPI/API (OFF)]{\includegraphics[width=0.99\linewidth]{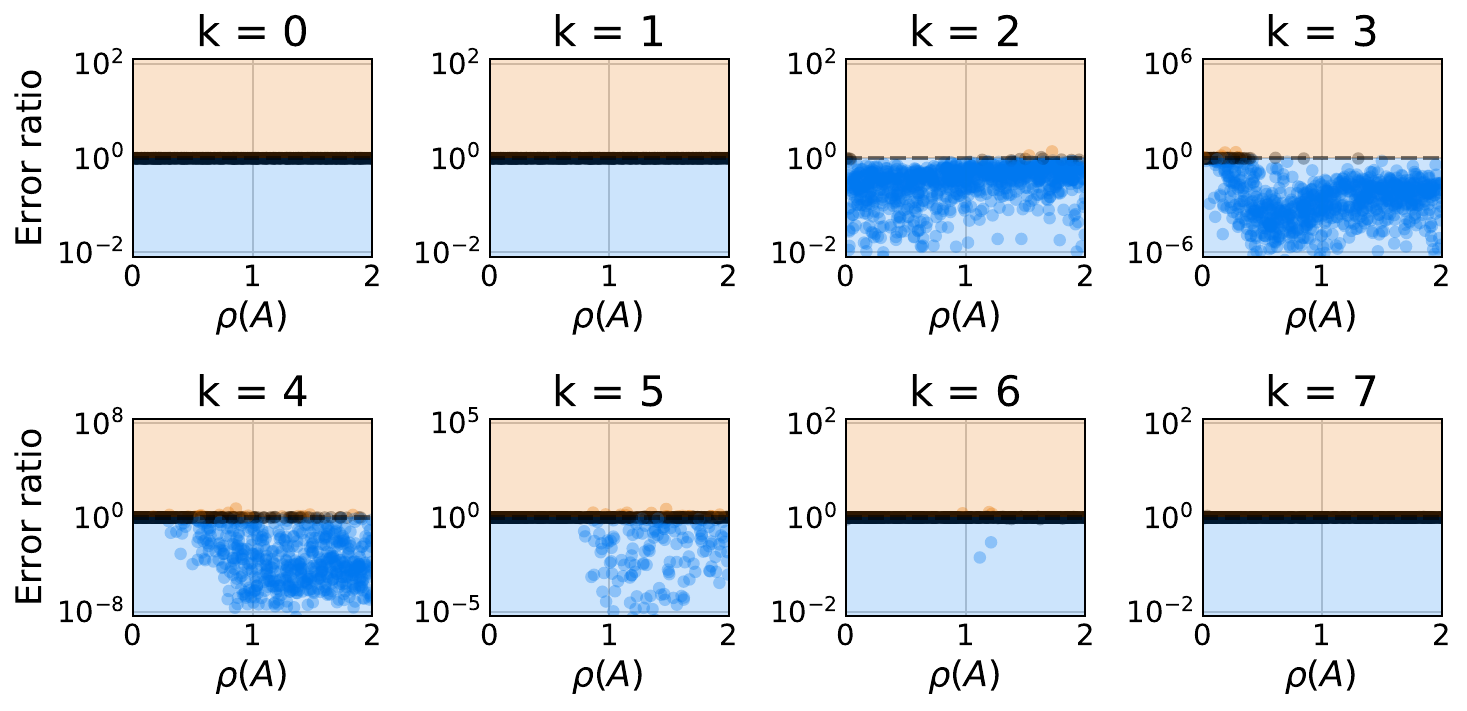}}
  }
\end{figure}

\begin{figure}[!htbp]
\floatconts
  {fig:plot_error_monte_carlo_false_scatter_on}
  {\caption{(a) Relative value error $\|P_k - P^*\|/\| P^* \|$ using AMPI and (b) ratio of relative error using AMPI divided by that using API, all with the offline algorithm variant (ON).}}
  {%
    \subfigure[Relative error using AMPI (ON)]{\includegraphics[width=0.99\linewidth]{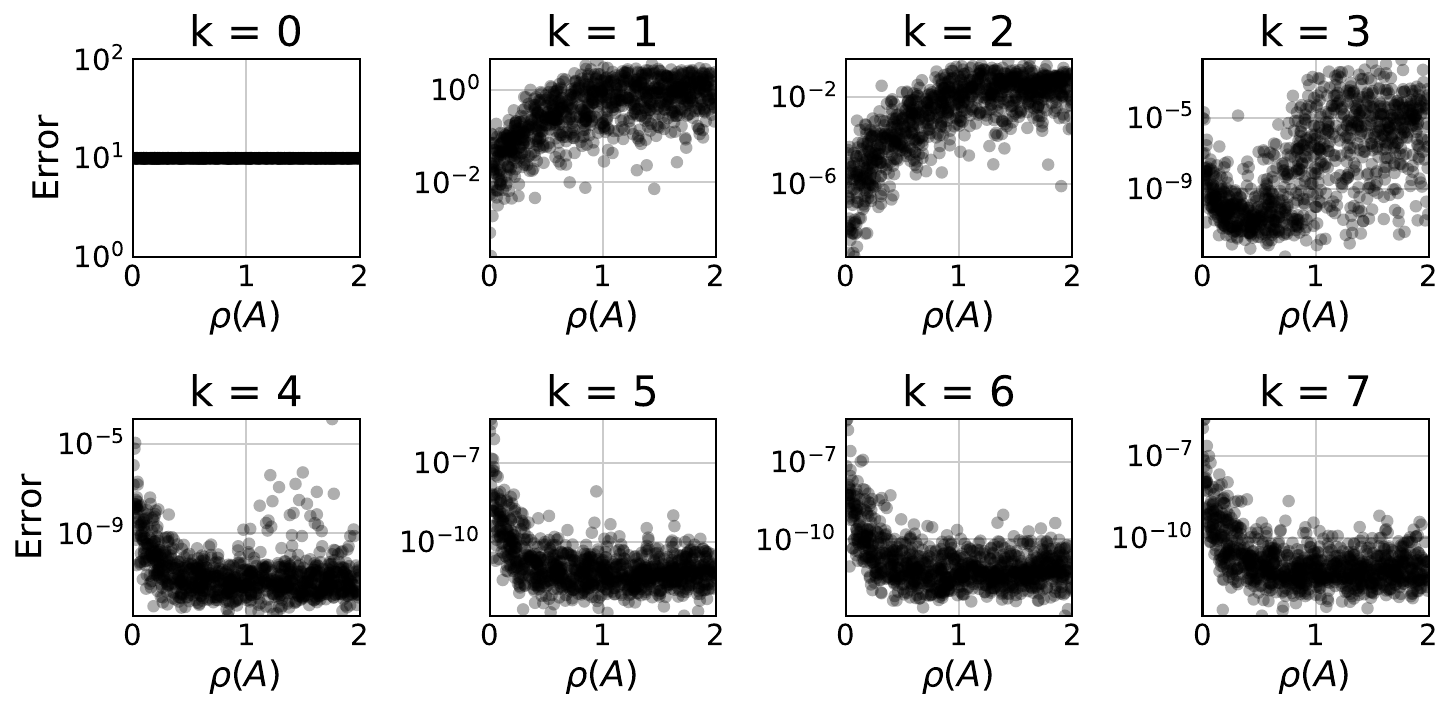}}\\ \ \\%
    \subfigure[Ratio of relative errors using AMPI/API (ON)]{\includegraphics[width=0.99\linewidth]{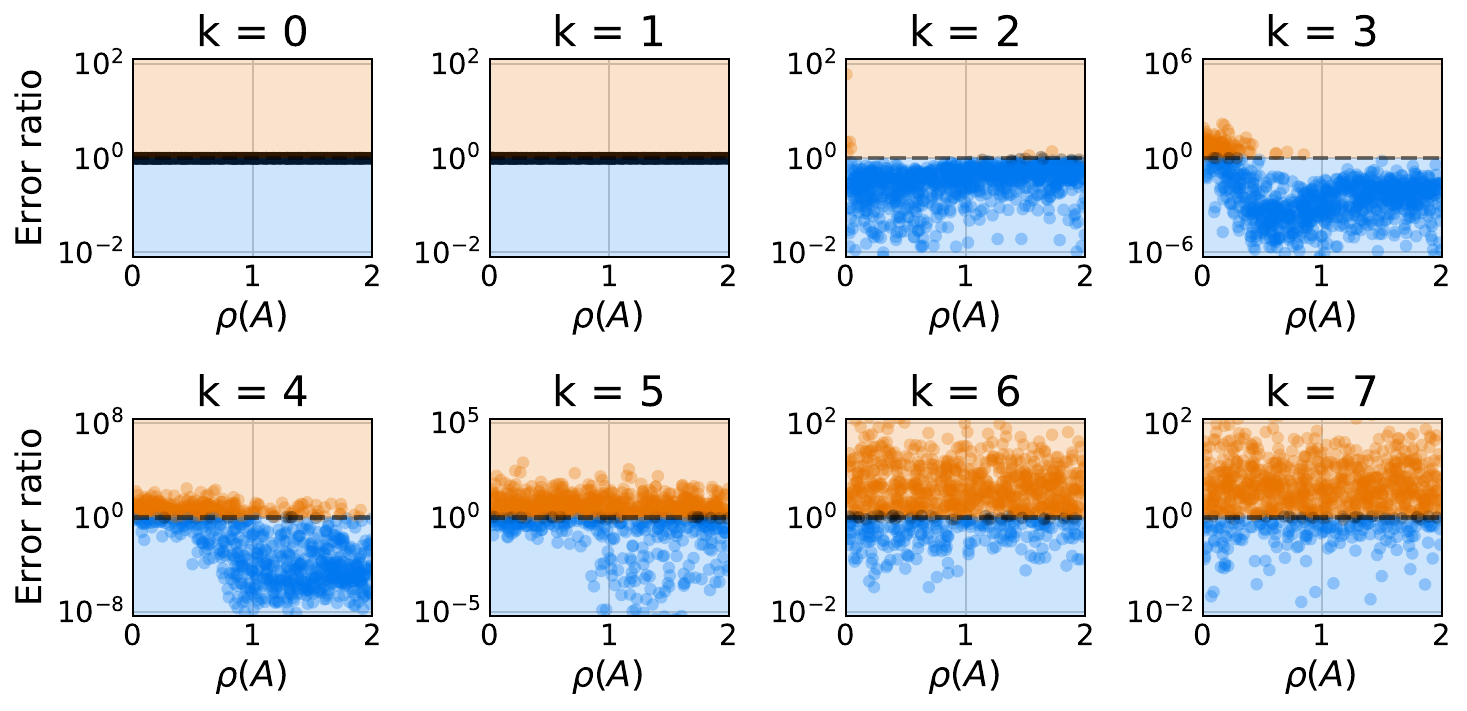}}
  }
\end{figure}

\clearpage
\section{Conclusions and future work}

Empirically, we found that regardless of the stabilizing initial policy chosen, convergence to the optimum always occurred when using the exact midpoint method. Likewise, we also found that approximate midpoint and standard PI converge to the same approximately optimal policy, and hence value matrix $P$, after enough iterations when evaluated on the same fixed off-policy rollout data $\cD$.
We conjecture that such robust, finite-data convergence properties can be proven rigorously, which we leave to future work.

This algorithm is perhaps most useful in the regime of practical problems in the online setting where it is relatively expensive to collect data and relatively cheap to perform the computations required to execute the updates. In such scenarios, the goal is to converge in as few iterations as possible, and MPI shows a clear advantage.
Both the exact and approximate midpoint PI incur a computation cost \emph{double} that of their standard PI counterparts. Theoretically, the faster \emph{cubic} convergence rate of MPI over the \emph{quadratic} convergence rate of PI should dominate this order constant (2$\times$) cost with sufficiently many iterations. However, unfortunately, due to finite machine precision, the total number of useful iterations that increase the precision of the optimal policy is limited, and the per-iteration cost largely counteracts the faster over-iteration convergence of MPI. This phenomenon becomes even more apparent in the model-free case where the ``noise floor'' is even higher. However, this disadvantage may be reduced by employing iterative Lyapunov equation solvers in Algorithm \ref{algo:exact_mpi} or iterative (recursive) least-squares solvers in Algorithm \ref{algo:approximate_mpi} and warm-starting the midpoint equation with the Newton solution.
Furthermore, the benefit of the faster convergence of the midpoint PI may become more important in extensions to nonlinear systems, where the order constants in Propositions \ref{prop:exact_convergence} and \ref{prop:approx_convergence} are smaller.

The current methodology is certainty-equivalent in the sense that we treat the estimated value functions as correct. Future work will explore ways to estimate and account for uncertainty in the value function estimate explicitly to minimize regret risk in the initial transient stage of learning when the amount of information is low and uncertainty is high.

\bibliography{bibliography.bib}
\end{document}